\documentclass[12pt]{amsart}
\usepackage{setspace}
\usepackage{amsfonts}
\usepackage{amsmath}
\usepackage{amssymb}
\usepackage{amsthm}
\usepackage{mathrsfs}
\usepackage[all]{xy}

\newcommand{\proofend}{\hfill \hbox{\vrule width 5pt height 5pt depth
0pt}}
\newcommand{\R}{\mathbb{R}}

\newcommand{\Z}{\mathbb{Z}}
\newcommand{\N}{\mathbb{N}}
\newcommand{\Q}{\mathbb{Q}}

\newcommand{\Oz}{\mathcal{O}}

\newcommand{\proj}{\mathbb{P}}
\newcommand{\A}{\mathbb{A}}

\newcommand{\F}{\mathbb{F}}

\newtheorem{thm}{Theorem}
\newtheorem{lemma}{Lemma}[section]
\newtheorem{corol}[thm]{Corollary}
\newtheorem{conj}[thm]{Conjecture}
\newtheorem{defin}{Definition}[subsection]

\begin{document}

\title[Silverman's conjecture for additive polynomials]{Silverman's conjecture for \\  additive polynomial mappings}
\author{Vesselin Dimitrov}
\address{Yale University Math. Dept. \\ 10 Hillhouse Avenue \\ CT~06520--8283 }
\email{vesselin.dimitrov@yale.edu}

\begin{abstract}
Let $F : \mathbb{G}_{a/K}^d \to \mathbb{G}_{a/K}^d$ be an additive polynomial mapping over a global function field $K/\F_q$, and let $P \in \mathbb{G}_a^d(K)$. Following Silverman, consider $\delta := \lim_{n \in \N} (\deg{F^{n}})^{1/n}$ the dynamic degree of $F$ and $\alpha(P) := \limsup_{n \in \N} h_K(F^{n}P)^{1/n}$ the arithmetic degree of $F$ at $P$. We have $\alpha(P) \leq \delta$, and extending a conjecture of Silverman from the number field case, it is expected that equality holds if the orbit of $P$ is Zariski-dense.

We prove a weaker form of this conjecture: if $\delta > 1$ and the orbit of $P$ is Zariski-dense, then also $\alpha(P) > 1$. We obtain furthermore a more precise result concerning the growth along the orbit of $P$ of the heights of the individual coordinates,  and formulate a few related open problems motivated by our results, including a generalization ``with moving targets''  of Faltings's theorem back in the number field case.
 \end{abstract}

\maketitle

\section{Introduction} \label{intro}

Silverman made the following conjecture in~\cite{silverman}, which for our purposes here we state only for the special but wide open case of polynomial mappings.

\begin{conj}[Silverman]  \label{silverconj}
Let $F : \A^d \to \A^d$ be a polynomial self-map over $\bar{\Q}$ and let $P \in \A^d(\bar{\Q})$ be an algebraic point with Zariski-dense forward orbit $(F^nP)_{n \in \N}$. Then,
\begin{equation} \label{equal}
\limsup_{n \in \N} h(F^nP)^{1/n} = \lim_{n \in \N} \, (\deg{F^{\circ n}})^{1/n}.
\end{equation}
\end{conj}

Here, $h: \A^d(\bar{\Q}) \to \R^{\geq 0}$ is the affine absolute logarithmic Weil height, and it is a general fact (cf.~Bellon-Viallet~\cite{algentropy}) that the limit on the right actually exists.
 When the mapping $F$ is generic in the sense that it avoids the zero locus in $\mathrm{Mor}(\A^d,\A^d)$ of the Macaulay multiresultant (a certain polynomial in the coefficients of $F$), then it extends as an endomorphism of projective space $\proj^d$. In that case, Conjecture~\ref{silverconj} is a simple consequence of the basic theory of heights; moreover, in that case, it suffices simply that the orbit of $P$ be infinite. The interest in the conjecture concerns the special mappings $F$: those belonging to the vanishing of the Macaulay multiresultant locus.

Silverman proved the upper bound in~(\ref{equal}), and hence this is a conjecture about a lower bound on the height along a Zariski-dense orbit for a polynomial iteration.
 In our preprint~\cite{vesselin} we used a simple arithmetic extrapolation procedure to obtain a very weak but non-trivial and completely general lower bound on orbit growth. The present paper rests on a similar principle, closer to traditional diophantine approximation proofs, to obtain a partial result on the positive characteristic counterpart of Conjecture~\ref{silverconj}.

While Silverman only considers the conjecture over a number field, it makes sense over an arbitrary  field with a notion of height. Over a global function field such as $\F_q(T)$ there is the rather interesting class of additive polynomial mappings, and for those the conjecture turns out to be more tractable.

\subsection{Notation} \label{notat}  let $C/\F_q$ be a regular, projective, and geometrically connected curve over $\F_q$. Let $K = K(C)$ be the function field of $C$. For $a \in K$ and a closed point $v \in C$, let $\mathrm{ord}_v(a)$ be the valuation of $a$ in $v$, and write $\mathrm{ord}_v^-(a) := \min(\mathrm{ord}_v(a),0)$. Then the affine logarithmic Weil height over $K$ on the vector group $\mathbb{G}_a^d$ is defined as
$$
h_K : \mathbb{G}_a^d(K) \to \R^{\geq 0}, \quad h_K(a_1,\ldots,a_d) := -\sum_{v \in C} \min_j \mathrm{ord}_v^-(a_j) \cdot \log{|k(v)|}.
$$
Applying the definition with $d = 1$, we have in particular a height of an element of $K$.
Further, for $G \in K[X_1,\ldots,X_d]$ a polynomial with coefficients in $K$, we define its height $h_K(G)$ to be the height of its set of coefficients, viewed as a $K$-valued point in the appropriate vector group $\mathbb{G}_a^N$.

Let $\tau : \mathbb{G}_a^d \to \mathbb{G}_a^d, \quad (x_1,\ldots,x_d) \mapsto (x_1^p, \ldots, x_d^p)$ be the $p$-power map, where $p$ is the characteristic of $\F_q$ (the radical of $q$). Recall that a polynomial mapping $F \in \mathrm{Mor}(\mathbb{G}_{a/K}^d,\mathbb{G}_{a/K}^d)$ meets the additivity constraint $F(x+y) = F(x) + F(y)$ if and only if it has the form
\begin{equation} \label{additive}
F = A_0 + A_1 \tau + \cdots + A_r \tau^r, \quad A_i \in M_d(K).
\end{equation}

\subsection{$t$-modules} \label{t-mod} Extending the standard terminology (Anderson~\cite{anderson}), by a \emph{$t$-module} of dimension $d$ over $K$ we simply mean an arbitrary ring homomorphism $\phi : \F_p[t] \to \mathrm{End}_{\F_p}(\mathbb{G}_{a/K}^d)$. (Usually one imposes the additional condition that all eigenvalues of $D\phi$ are equal, for reasons not needed in this paper.) Via $\phi(t) := F$, extended by linearity and composition ($\phi(t^n) = F^{\circ n}$), this is equivalent to the datum of an additive polynomial mapping $F : \mathbb{G}_{a/K}^d \to \mathbb{G}_{a/K}^d$ as above. A \emph{sub $t$-module} is then a connected algebraic subgroup of $\mathbb{G}_{a/K}^d$ invariant under $\phi$, or equivalently, under $F$.

\subsection{The dynamic and arithmetic degrees} \label{notass} Those are defined respectively by $\delta := \lim_{n \in \N} (\deg{F^n})^{1/n}$, and for $P \in \mathbb{G}_a^d(K)$ a $K$-point, by $\alpha(P) := \limsup_{n \in \N} h_K(F^nP)^{1/n}$. Further, for $\lambda \in K(\mathbb{G}_{a/K}^d)$ a non-constant rational function defined on the orbit $(F^nP)_{n \in \N}$, we let the corresponding restricted degrees
$$
\delta_{\lambda} := \limsup_{n \in \N} (\deg \lambda \circ F^n)^{1/n} \leq \delta
 $$
 and
 $$
 \alpha_{\lambda}(P) := \limsup_{n \in \N} h_K(\lambda(F^nP))^{1/n} \leq \alpha(P).
 $$
  We have $\alpha(P) \leq \delta$ by the argument in Prop.~12 in Silverman~\cite{silverman}. More generally, we will see (Corollary~\ref{restupper}) that $\alpha_{\lambda}(P)
\leq \delta_{\lambda}$.

Silverman's Conjecture~1 in~\cite{silverman} predicts that the arithmetic degrees $\{ \alpha(P) \in P \in \mathbb{G}_a^d(K)\}$ make up a finite set of algebraic integers, with $\alpha(P) = \delta$ for all $P$ having Zariski-dense orbit. In our situation, we can be more precise:

\begin{conj} \label{addconj}
Let as above $F$ and $\lambda$ be, respectively, a non-zero additive mapping and a non-constant rational function on $\mathbb{G}_{a/K}^d$. Then:
\begin{itemize}
\item[(i)]  The dynamic degrees $\delta$ and $\delta_{\lambda}$ are rational powers of $p$.
\item[(ii)] If $P \in \mathbb{G}_a^d(K)$ has $\lambda$ defined along $(F^nP)_{n \in \N}$ and $h_K(\lambda(F^nP)) = o\big(\deg{(\lambda \circ F^{\circ n})}\big)$, then $P$ belongs to a translate by a preperiodic point of a proper connected algebraic subgroup of $\mathbb{G}_{a/K}^d$ stable under $F$.
\item[(iii)]  In particular, if $P$ does not belong to such a translate, then $\alpha(P) = \delta$ and $\alpha_{\lambda}(P) = \delta_{\lambda}$.
\end{itemize}
\end{conj}

Here, a point is said to be preperiodic if it has finite orbit under $F$. In the language of (extended) $t$-modules introduced above, the special subvarieties in (ii, iii) are exactly the torsion translates of proper sub $t$-modules of the $t$-module $\phi$ defined by $F$. In the case that $\phi$ is a pure $\F_p[t]$-module of weight $w \in \Q^{> 0}$ in the sense of Anderson~\cite{anderson}, which is the object in function field arithmetic most closely analogous to an abelian variety, we expect more precisely that $\delta = p^{1/w}$ in (i).

\subsection{The main result} While we are unable to establish the equality $\alpha(P) = \delta$ in Conjecture~\ref{addconj} (iii), we at least show that $\delta > 1$ yields also $\alpha(P) > 1$. More generally, for $\lambda : \mathbb{G}_{a/K}^d \to \mathbb{G}_{a/K}$ any non-zero additive regular function, we establish the corresponding assertion for the restricted degrees $\delta_{\lambda}$ and $\alpha_{\lambda}(P)$.

\begin{thm} \label{main}
Let $F \in \mathrm{Hom}_{\F_p}(\mathbb{G}_{a/K}^d,\mathbb{G}_{a/K}^d)$ be an additive polynomial mapping and $\lambda \in \mathrm{Hom}_{\F_p}(\mathbb{G}_{a/K}^d,\mathbb{G}_{a/K})$ an additive polynomial function, both assumed non-zero. Consider a $K$-point $P \in \mathbb{G}_{a}^d(K)$ not belonging to a translate by a preperiodic point of a proper connected algebraic subgroup of $\mathbb{G}_{a/K}^d$ stable under $F$. If $\delta_{\lambda} > 1$, then also $\alpha_{\lambda}(P) > 1$.
\end{thm}

\medskip

The proof of Theorem~\ref{main} follows a principle of diophantine approximation, and rests on Yu's $t$-module zero estimate~\cite{yu}.
Taking for $\lambda$ the coordinate projections we record the obvious corollary.

\begin{corol}  \label{reccor}
Assume as above $F$ is an additive endomorphism of $\mathbb{G}_{a/K}^d$, and consider $P \in \mathbb{G}_a^d(K)$.
 If $\delta > 1$ but $\alpha(P) = 1$, then the orbit $(F^nP)_{n \in \N}$ is not Zariski-dense.
\end{corol}

In section~\ref{cmpl} below we prove a strengthening of Corollary~\ref{reccor}, including a proof of the equality $\alpha(P) = \delta$ for additive group endomorphisms of small dynamic degree $\delta$, under an unproved algebraic hypothesis on $F$ that is likely to be satisfied for all  $t$-modules.

 \subsection{Further problems}
 An analog for linear tori of Conjecture~\ref{addconj} (ii, iii) follows from W.M. Schmidt's Subspace theorem: If $\lambda \in \bar{\Q}(\mathbb{G}_m^r)$ is a non-constant rational function and $\boldsymbol{\alpha} \in \mathbb{G}_m^r(\bar{\Q})$ has $h(\lambda(\boldsymbol{\alpha}^n)) = o(n)$, then $\boldsymbol{\alpha}$ belongs to a torsion translate of a proper linear subtorus. More precisely, $h(\lambda(\boldsymbol{\alpha}^n)) = o(n)$ holds if and only if there is a homomorphism $p : \mathbb{G}_m^r \to \mathbb{G}_m^s$ and a $\mu \in \bar{\Q}(\mathbb{G}_m^s)$ such that $p(\boldsymbol{\alpha})$ is torsion and $\lambda = p^* \mu$.

We close the introduction by raising as a conjecture  the analogous statement for abelian varieties. We do it in two versions, the first being the direct translation of the $\mathbb{G}_m^r$ statement just given:

\begin{conj} \label{elliptic}
Let $A$ be an abelian variety and  $\lambda : A \dashrightarrow \proj^1$ a rational function, both defined over $\bar{\Q}$. Consider $P \in A(\bar{\Q})$, all of whose positive integer multiples $[n]P$ belong to the domain of definition of $\lambda$. Then $h(\lambda([n]P)) = o(n^2)$ as $n \to +\infty$ if and only if there exists a homomorphism $p : A \to B$ to an abelian variety $B$ and a rational function $\mu : B \dashrightarrow \proj^1$ such that $\lambda = p^*\mu$ and $p(P) \in B$ is a torsion point.
\end{conj}

A consequence of this would be that unless $P$ lies in a torsion translate of a proper abelian subvariety of $A$, all coordinate functions $X_i/X_0$ in a fixed projective embedding of $A$ have a comparable arithmetic complexity along the the sequence of positive integer multiples $[n]P$.

The second version of our conjecture concerns a $\Z^d$-action, and is much stronger. Taking into account the finite generation of the Mordell-Weil group, we give it in a form that makes no reference to the action at all.

\begin{conj} \label{faltmov}
Let $A$ be an abelian variety and $\lambda : A \dashrightarrow \proj^1$ a rational function, both defined over a number field $K$. Let $h$ be the standard height on $\proj^1$ and $\widehat{h}$ the canonical height associated with a fixed symmetric projective embedding of $A$. Then,
\begin{equation}  \label{infim}
\inf_P \frac{1+h(\lambda(P))}{\widehat{h}(P)} > 0
\end{equation}
as $P$ ranges over the $K$-rational points $P \in A(K) \cap \mathrm{dom}(\lambda)$ not lying in a translate $Q + B$ of a positive-dimensional abelian subvariety of $A$ that is contracted by $\lambda$, i.e., $\lambda(Q+B) = \lambda(Q)$.

In particular, if $A$ is simple and $\lambda$ non-constant, then~(\ref{infim}) holds over all $K$-rational points $P \in A(K) \cap \mathrm{dom}(\lambda)$ in the domain of definition of $\lambda$.
\end{conj}

 It is easily seen that this conjecture refines the statement of Faltings's ``big theorem''~\cite{faltings}. If we consider the level sets $\lambda(P) = a$ as a varying hypersurface,  Conjecture~\ref{faltmov} can be seen as a ``Mordell-Lang conjecture with moving targets,'' in analogy with Vojta's Roth theorem with moving targets~\cite{rothmov}.

\section{An estimate for the coefficients of $F^{\circ n}$}

Continuing the notation and assumptions from~\ref{notass}, we let $F$ and $\lambda$ respectively a non-zero additive polynomial self-map and a non-constant rational function on $\mathbb{G}_{a/K}^d$.

\begin{lemma} \label{truncbound}
Let $1 \leq s < s_+$ be real numbers. For $n \gg_{s,s_+} 0$, all monomials in $F^{\circ n}$ with degree at most $s^n$ have height at most $s_+^n$.
\end{lemma}

\begin{proof}
Considering the expression of $F$ in the form~(\ref{additive}), it is enough to prove that the coefficient matrix of $\mathbf{x}^{p^i}$ in $F^{\circ n}$ has entries of height bounded by $Cnp^i$, where $C$ is an upper bound for the heights of the entries of the matrices $A_0,\ldots,A_r$. We do this by induction on $n$, the base being clear. For the induction step, write
$$
F^{\circ n} = A_{n,0} + A_{n,1} \tau + \cdots A_{n,rn} \tau^{rn} \quad \textrm{ with } h_K(A_{n,i}) \leq Cnp^i,
$$
and note that the matrix coefficient of $\mathbf{x}^{p^i}$ in $F^{\circ (n+1)} = F \circ F^{\circ n}$ is
\begin{equation} \label{recrs}
\sum_{0 \leq j \leq i} A_j A_{n,i-j}^{p^j},
\end{equation}
where the undefined $A$'s are understood to be zero.
\end{proof}

As an immediate consequence, applying the lemma with $s := \delta_{\lambda} + \epsilon$ and letting $\epsilon \to 0$, we conclude the upper bound in Conjecture~\ref{addconj} (iii).

\begin{corol}  \label{restupper}
For every $P \in \mathbb{G}_{a/K}^d(K)$ it holds $\alpha_{\lambda}(P) \leq \delta_{\lambda}$. \proofend
\end{corol}

\section{The auxiliary construction}  \label{auxiliary}

 We assume that $\delta_{\lambda} > 1$ and $\alpha_{\lambda}(P) = 1$,
  and set out to prove that $P$ lies in a torsion translate of a proper sub $t$-module of $(\mathbb{G}_{a/K}^d,F)$.
   To this end, we introduce parameters $\delta_1, \ldots, \delta_4$, and $\delta_+$ with
 \begin{equation} \label{pamchoix}
 1 < \delta_4 < \delta_3 < \delta_2 < \delta_1 < \delta_{\lambda} < \delta_+
 \end{equation}
 and
 \begin{equation}  \label{concav}
 \delta_2^{d+1} < \delta_1^d\delta_3, \quad \delta_1 < \delta_2\delta_4.
 \end{equation}
 In what follows we let $\Pi_N := (F^N)^*\lambda$ and consider a sequence of $N$ going to infinity such that $\deg{\Pi_N} > \delta_1^N$. All asymptotics in $N$ will be with respect to a fixed choice of $\delta$ parameters.

 Let $\mathcal{U}_N$ be the set of monomials in $\mathbf{x} := (x_1,\ldots,x_d)$ (the coordinate functions of $\mathbb{G}_a^d$) of degree $< \delta_1^N$. We have
 \begin{equation}  \label{monomfreed}
 |\mathcal{U}_N| \sim \delta_1^{dN}/d! \gg_d \delta_1^{dN}.
 \end{equation}

The idea in this paper is to construct, and exploit, a non-zero polynomial $G = G_{N} \in K[\mathbf{x}]$,
\begin{equation}  \label{auxifun}
  G_N(\mathbf{x}) = \sum_{\substack{u \in \mathcal{U}_N \\ \ell < \delta_4^{dN}}} c_{u,\ell} \cdot u \Pi_N^{\ell},
\end{equation}
with  $c_{u,\ell} \in K$ of small height, vanishing to  order $(\delta_2\delta_4)^N$ at $\mathbf{x} = \mathbf{0}$.

We fix a closed point $\infty \in C$ and declare it to be the ``place at infinity.''

\begin{lemma} \label{siegel}
For $N \gg 0$, there exist $c_{u,\ell} \in A := \Gamma(C - \infty,\Oz_C)$, not all zero, such that
\begin{equation} \label{estaux}
h_K(c_{u,\ell}) < (\delta_3\delta_4)^N \quad \textrm{ for all } u \in \mathcal{U}_N, \, \ell < \delta_4^{dN}
\end{equation}
and the function $G_N$ defined by~(\ref{auxifun}) vanishes to order at least $(\delta_2\delta_4)^n$ at $\mathbf{x} = \mathbf{0}$.
\end{lemma}

\begin{proof}
Immediate from  Siegel's lemma for function fields (cf. Thunder~\cite{thunder}), using Lemma~\ref{truncbound} and the first inequality in our constraint (\ref{concav}). The number of free parameters in the linear system is $L \asymp_d (\delta_1\delta_4)^{dN}$, the number $M$ of linear equations is fewer than $(\delta_2\delta_4)^{dN}$, the Dirichlet exponent $M/(L-M)$ is $\asymp_d \big( \delta_2/\delta_1 \big)^{dN}$, and the height of the linear system is $\ll (\delta_2\delta_4 + o(1))^{N}$.
\end{proof}

\section{The extrapolation} \label{extrapolate}

\subsection{A reduction} \label{reduxio} We fix a closed point $v_0$ of $C$ with $v_0 \neq \infty$ such that $P \in \mathbb{G}_a^d(\mathcal{O}_{C,v_0})$ and $F$ extends as an endomorphism of the additive group $\mathbb{G}_a^d$ over $\mathcal{O}_{C,v_0}$. In other words, denoting by $v_0$ also the corresponding place as well as normalized valuation of $K$, we require the $v_0$-integrality of $P$ and of all matrix coefficients of the $A_i$ in~(\ref{additive}).

Then $F^sP \in \mathbb{G}_a^d(\mathcal{O}_{C,v_0})$ for all $s \in \N$, and since the residue field $k(v_0)$ is finite, it follows by the pigeonhole principle that there are $s_1 < s_2$ such that all components of $(F^{s_2}-F^{s_1})P$ belong to the maximal ideal $\mathfrak{m}_{v_0}$ of $\mathcal{O}_{C,v_0}$.

It is clear that $Q := (F^{s_2}-F^{s_1})P$ is $F$-preperiodic if and only $P$ is. Moreover, due to the additivity of $F$, it certainly holds $\alpha(Q) \leq \alpha(P)$. Hence, in proving Theorem~\ref{main}, we may replace $P$ with $Q$, and may thus make the following assumption:
\begin{equation} \label{redux}
\textrm{All components of $P$ belong to $\mathfrak{m}_{v_0}$. }
\end{equation}

\subsection{Additivity and upper bound}
 Our next task consists of showing that, for $N \gg 0$, the assumption $\alpha_{\lambda}(P) = 1$ and the constraints~(\ref{pamchoix}) and~(\ref{concav}) imply that $G_N$ vanishes at the new points $Q := b(F)(P)$, for a certain range $\deg{b} < cN$ of polynomials $b \in \F_p[t]$. It is at this step that the additivity constraint on $F$ and $\lambda$ is used crucially. Under our assumption $\alpha_{\lambda}(P) = 1$, it implies
 \begin{equation}  \label{expl}
  h_K(\Pi_N(Q)) < \delta_4^{N+\deg{b}} < \delta_4^{(1+c)N}, \quad \quad N \gg 0.
 \end{equation}
 By Lemma~\ref{siegel} we obtain the upper bound
 \begin{equation}  \label{uppbnd}
 h_K(G_N(Q)) \ll_{F,\lambda,P} (\delta_3\delta_4)^N + \delta_1^N + \delta_4^{(d+1+c)N}.
 \end{equation}

 \subsection{The lower bound} By construction, denoting $I$ the maximal ideal $(x_1,\ldots,x_d)$ in the ring $\Oz_{C,v_0}[\mathbf{x}]$, the polynomial $G_N(\mathbf{x}) \in K[\mathbf{x}]$ actually lies in
 $$
 G_N(\mathbf{x}) \in  I^{T} \cdot \Oz_{C,v_0}[\mathbf{x}], \quad \quad T := \lfloor (\delta_2\delta_4)^N \rfloor
 $$
Since on the other hand our reduction~(\ref{redux}) implies for any $b \in \F_p[t]$ that $Q = b(F)(P) \in \mathbb{G}_a^d(\mathfrak{m}_{v_0})$ has components in $\mathfrak{m}_{v_0}$, we conclude that
\begin{equation} \label{lowerbnd}
\mathrm{ord}_{v_0} G_N(Q) ) \geq T > (\delta_2\delta_4)^N - 1.
\end{equation}

\subsection{Comparison} We introduce the further restriction
\begin{equation}  \label{epsa}
\delta_2 > \delta_4^{d+1+c}
\end{equation}
 on our parameters.
Comparing the bounds~(\ref{uppbnd}) and~(\ref{lowerbnd}) we obtain for $N \gg 0$ the vanishing $G_N(b(F)P) = 0$ in the range $b \in \F_p[t]$, $\deg{b} < cN$. This completes the extrapolation.

\section{The zero estimate and conclusion  \\of the proof of Theorem~\ref{main}}  \label{zeroestimate}

\subsection{Yu's theorem}  We state, in the language of additive polynomial mappings (cf.~section~\ref{t-mod} above for the translation to $t$-modules), a particular case of Yu's zero multiplicity estimate. For the proof of Theorem~\ref{main} we only need it in the relatively easier case with no multiplicities ($U = 0$).

For $S \in \N$, write
$$
\Gamma(S,P) = \{ b(F)(P) \mid b \in \F_p[t], \, \deg{b} \leq S \} \subset K
$$
 for the $\F_p$-linear span of the points $F^mP$, $m = 0, \ldots,S$.

\begin{lemma}[Yu~\cite{yu}]  \label{yulem}
Consider (as before) $F : \mathbb{G}_{a/K}^d \to \mathbb{G}_{a/K}^d$ an additive group endomorphism. There is a constant $C = C(F) < \infty$ such that the following is true. Suppose $G \in K[\mathbb{G}_{a/K}^d] = K[\mathbf{x}]$ is a non-zero polynomial of total degree at most $D$ vanishing to order at least $dU + 1$ at each point in $\Gamma(R,P)$. Then, there exists a proper connected algebraic subgroup $H \lneq \mathbb{G}_{a/K}^d$, stable under $F$, such that
\begin{equation}  \label{zme}
\binom{U + \mathrm{codim}{H}}{\mathrm{codim}{H}} \cdot \deg{H} \cdot \# ( \Gamma(R-d+1,P) / H )  \leq C D^{\mathrm{codim}{H}}.
\end{equation}
\end{lemma}

\begin{proof}
This is Theorem~2.2 of Yu~\cite{yu} with $l = 1$ in {\it loc. cit.} and $\Phi = \mathbb{G}_{a/K}^d$ the tautological analytic submodule. Yu works in the context of Anderson $t$-modules involving the additional constraint 1~(ii) in {\it loc. cit.}, which however is not used in the proof of Theorem~2.2.
\end{proof}

\subsection{Conclusion}

Noting that $\# \Gamma(cN - d + 1,P) = p^{\lfloor cN \rfloor - d + 1}$ and that $\deg{G_N} < (\delta_+\delta_4^d)^N < (\delta_+\delta_{\lambda}^d)$ for $N \gg 0$, we now introduce the final condition
\begin{equation}  \label{finas}
p^{c} > (\delta_+ \delta_{\lambda}^d \big)^{d}
\end{equation}
on our parameters. The conditions (\ref{pamchoix}), (\ref{concav}), (\ref{epsa}), and (\ref{finas})  are compatible. We first choose $\delta_+ > \delta_{\lambda}$ arbitrarily. Next we choose $c$ large enough for~(\ref{finas}) to hold. Then, by taking $\delta_4 > 1$ sufficiently close to $1$, we ensure that~(\ref{epsa}) holds with a $\delta_2 < \delta_{\lambda}$. Finally, we choose $\delta_1 \in (\delta_2,\delta_{\lambda})$ and $\delta_3 \in (\delta_4,\delta_2)$ such that both inequalities in~(\ref{concav}) are satisfied.

 Applying Lemma~\ref{yulem} with $D := (\delta_+\delta_{\lambda})^N$ and $U := 0$ we conclude that there is a non-zero $b \in \F_p[t] \setminus \{0\}$ and a proper connected algebraic subgroup $H \lneq \mathbb{G}_{a/K}^d$ stable under $F$ and containing $b(F)(P)$. This means that the $F$-orbit of $P$ lies in a torsion translate of a proper sub $t$-module. \proofend

\section{Complements}  \label{cmpl}

In this section we explain how Corollary~\ref{reccor} can be improved under an additional purely algebraic hypothesis that is likely to be satisfied for all $t$-modules.

\subsection{Saturation degree} \label{algcap} This is a notion arising naturally in our context, concerning the possibility of using more than one coordinate function in the auxiliary construction. We define it more generally for an arbitrary morphism (polynomial self-map) $F : \A_K^d \to \A_K^d$ over any field $K$. Let $m_1,\ldots,m_d : \A_K^d \to \A_K^1$ be the coordinate projections.

\begin{defin}
The \emph{saturation degree} of $F$ is the supremum $\kappa(F)$ of the positive real values $\kappa_-$ for which there exists a constant $L < \infty$ such that, for all $N \gg_{L,\kappa_-} 0$, the $K$-linear span of all products of the regular functions
\begin{eqnarray*}
\Psi_{i,n,j} :=  \big( (F^n)^* m_i \big)^j
\end{eqnarray*}
over subsets of $V := \big\{ 1 \leq i \leq d, \, 0 \leq n < N, \, 0 \leq j < L \big\}$, has dimension exceeding $\kappa_-^{dN}$.
\end{defin}

Clearly, $\delta^{1/d} \leq \kappa \leq \delta$. The following basic example shows that the right inequality can be strict and the left inequality can be an equality. If $F = (f_1(x_1), \ldots, f_d(x_d))$ with $f_j(x) \in K[x]$ univariate polynomials of degrees $q_j := \deg{f_j} > 0$, then $\kappa = (q_1 \cdots q_d)^{1/d}$, whereas $\delta = \max_{j=1}^d q_j$.

Nevertheless it seems that $\kappa = \delta$  holds outside of degenerate situations. To make a specific statement relevant to the goals of this paper we return to the case of additive mappings.

\subsection{The Saturation Conjecture}  \label{satco}

The following conjecture, if true, would determine the $\kappa$ invariant for all $t$-modules.

\begin{conj} \label{capconj}
(i) Suppose as before that $F \in \mathrm{End}_{\F_p}(\mathbb{G}_{a/K}^d)$ an additive polynomial mapping. If no non-zero proper connected algebraic subgroup of $\mathbb{G}_{a/K}^d$ is stable under $F$, then the saturation degree of $F$ equals its dynamic degree: $\kappa = \delta$. (ii) $\kappa^d$ is multiplicative over short exact sequences of $t$-modules: given $0 \to H \to G \to G/H \to 0$, it holds $\kappa(G)^{\dim{G}} = \kappa(H)^{\dim{H}} \kappa(G/H)^{\dim{G/H}}$.
\end{conj}

\medskip

{\it Example. } If the matrix coefficient $A_r$ in the presentation~(\ref{additive}) is invertible, meaning equivalently that $F$ extends as a morphism $\proj_K^d \to \proj_K^d$, then the condition $\kappa = \delta$  is satisfied. \proofend

\subsection{A conditional result}

We can now state our complement to Theorem~\ref{main}.

\begin{thm} \label{suppl}
Assume the Saturation Conjecture~\ref{capconj}. Let
$$
F \in \mathrm{Mor}_{\F_p}(\mathbb{G}_{a/K}^d,\mathbb{G}_{a/K}^d)
$$
 be an additive polynomial mapping for which $\kappa = \delta$, for example, any simple $t$-module. Consider a $K$-point $P \in \mathbb{G}_{a}^d(K)$. Then, if $P$ is not preperiodic under $F$, it holds
\begin{equation}  \label{lwrbndimp}
\delta \geq \alpha(P) \geq \min(\delta, p^{1/d}).
\end{equation}
\end{thm}

{\it Remark. } In fact, the method of proof yields somewhat more. Either $P$ is preperiodic, or else for any finite field extension $\F_{q_0}/\F_p$ it holds $\max_{\zeta \in \F_{q_0}^{\times}} \alpha(\zeta P) \geq \min(\delta, q_0^{1/d})$. In particular, if $F$ is $\F_{q_0}$-linear, the inequality~(\ref{lwrbndimp}) improves naturally to $\alpha(P) \geq \min(\delta,q_0^{1/d})$.

\subsection{Small dynamic degree}  Thus, under our Saturation Conjecture, we obtain Silverman's conjectured equality of arithmetic and dynamic degrees for simple $t$-modules of dynamic degree not exceeding $p^{1/d}$. The example of the $d$-th tensor power $C^{\otimes d}$ of the Carlitz module demonstrates that it is possible to have $\delta = p^{1/d}$. We do not know if $\delta \leq p^{1/d}$ can be strictly satisfied (with $1 < \delta < p^{1/d}$); to the contrary, it seems possible that the rational number $\log{\delta} / \log{p}$ in Conjecture~\ref{addconj} (i) always has denominator not exceeding $d$.

\medskip

The proof of Theorem~\ref{suppl} is a variant of that of Theorem~\ref{main}, using more freedom in the auxiliary construction and also multiplicities in the extrapolation. We indicate the changes in the argument.

\subsection{Proof: Auxiliary construction} We assume
\begin{equation} \label{contpo}
\alpha < \min(\delta,p^{1/d})
\end{equation}
and consider similarly constrained $\delta$ parameters as before:
\begin{equation}  \label{deltprim}
\alpha < \delta_4 < \delta_3 < \delta_2 < \delta_1 < \delta < \delta_+, \quad \delta_2^{d+1} < \delta_1^d\delta_3.
\end{equation}
By assumption, there is a subset $S \subset 2^V$ such that, for $N \gg_{\delta_1} 0$,
\begin{equation} \label{fred}
|S| > \delta_1^{Nd} \textrm{ and $\prod_{(i,n,j) \in s}\Psi_{i,n,j}, \quad s \in S$, \, are linearly independent.}
\end{equation}
Our auxiliary construction will now take the form
\begin{equation}  \label{newux}
G_N(\mathbf{x}) = \sum_{s \in S} c_s \prod_{(i,n,j) \in s} \Psi_{i,n,j}, \quad c_s \in K.
\end{equation}
Siegel's lemma ensures $c_s \in \Gamma(C - \infty, \Oz_C)$, not all zero, with
\begin{equation}  \label{newsi}
h_K(c_s) < \delta_3^N,
\end{equation}
such that the function $G_N$ defined by~(\ref{newux}) vanishes to order at least $T  := \lfloor \delta_2^N \rfloor$ at $\mathbf{x} = \mathbf{0}$.

\subsection{Hyperderivatives}
 Our next task consists of showing that, for $N \gg 0$, our assumption~(\ref{contpo}) implies that $G_N$ vanishes to order $\delta_4^N$ at the new points $Q := b(F)(P)$, for a certain range $\deg{b} < cN$ of polynomials $b \in \F_p[t]$.

 A polynomial $G \in K[\mathbf{x}]$ vanishes to order $U$ at $\mathbf{x} = Q$ if and only if  the expansion of $G(\mathbf{z} + Q)$ contains no monomials of degree $\leq U$. Introducing the \emph{hyperderivative}
 \begin{equation}
 \Delta_{\mathbf{i}}(G)(Q) := [\mathbf{z}^{\mathbf{i}}] \, Q(\mathbf{z} + Q),
 \end{equation}
 this means having $\Delta_{\mathbf{i}}(G)(Q) = 0$ for all $\mathbf{i} = (i_1,\ldots,i_d)$ with $|\mathbf{i}| := \sum_{t=1}^d i_t \leq U$. Note that if $G$ vanishes at $\mathbf{x} = \mathbf{0}$ to order $U \geq |\mathbf{i}|$, then $\Delta_{\mathbf{i}}(G)$ does to order $U - |\mathbf{i}|$.

 The equality $\Delta_{\mathbf{i}}(G_N)(Q) = 0$ in our extrapolation will be obtained from the product formula in $K$. We start by making the reduction~\ref{reduxio}. Supposing $\Delta_{\mathbf{i}}(G_N)(Q) \neq 0$, the product formula yields (cf.~\ref{notat})
 \begin{eqnarray} \label{prdn}
 \label{pform} \mathrm{ord}_{v_0} (\Delta_{\mathbf{i}}(G_N)(Q) )  \log{|k(v_0)|} \\ \label{pform2} = -\sum_{v \neq v_0} \mathrm{ord}_{v}(\Delta_{\mathbf{i}}(G_N)(Q) )  \log{|k(v)|}
  \leq h_K(\Delta_{\mathbf{i}}(G_N)(Q)).
 \end{eqnarray}

Our aim now is to derive a contradiction by estimating the left-hand side~(\ref{pform}) from below, using~(\ref{redux}) and the high order of vanishing at $\mathbf{x} = \mathbf{0}$,  and the right-hand side from above, using Lemma~\ref{siegel} and the additivity of $F$.

\subsection{The lower bound} By construction, denoting $I$ the maximal ideal $(x_1,\ldots,x_d)$ in the ring $\Oz_{C,v_0}[\mathbf{x}]$, the polynomial $\Delta_{\mathbf{i}}(G_N)(\mathbf{x}) \in K[\mathbf{x}]$ actually lies in
$$
\Delta_{\mathbf{i}}(G_N)(\mathbf{x}) \in I^{T - |\mathbf{i}|} \cdot \Oz_{C,v_0}[\mathbf{x}].
$$
Since on the other hand our reduction~\ref{reduxio} implies for any $b \in \mathbb{F}_p[t]$ that $Q := b(F)(P)$ has components in $\mathfrak{m}_{v_0}$, we conclude that
\begin{equation} \label{lwrb}
\mathrm{ord}_{v_0} (\Delta_{\mathbf{i}}(G_N)(Q)) \geq T - |\mathbf{i}|.
\end{equation}

\subsection{The upper bound} Here the additivity of $F$ and the boundedness by $L = L(\delta_1)$ of the exponents $j$ in the considered $\Psi_{i,n,j}$ are crucial. We have
\begin{eqnarray*}
\Psi_{i,n,j}(\mathbf{z}+Q) = m_i(F^n(\mathbf{z}+Q))^j = (m_i(F^n(\mathbf{z}))+m_i(F^n(Q)))^j,
\end{eqnarray*}
and the heights of the $K[\mathbf{z}]$-coefficients in this expansion are estimated by Lemma~\ref{truncbound} using $h_K(F^n(Q)) = h_K(F^nb(F)P) = O(\delta_4^{n+\deg{b}})$.
From~(\ref{newsi}), we thus estimate
\begin{equation} \label{uprb}
h_K(\Delta_{\mathbf{i}}(G_N)(Q)) = O(\delta_3^N + \delta_4^{(1+c)N}), \quad |\mathbf{i}| \leq \delta_4^N, \quad \deg{b} < cN.
\end{equation}

\subsection{Comparison} Introducing now the constraint
\begin{equation}  \label{compr}
\delta_2 > \delta_4^{1+c},
\end{equation}
the comparison of (\ref{lwrb}) with (\ref{uprb}) negates~(\ref{prdn}), forcing in our range $\deg{b} < cN$ the vanishing of $G_N$ to order $\delta_4^N$ at all points $b(F)(P)$.

\subsection{Proof of Theorem~\ref{suppl}} We now apply Lemma~\ref{yulem} with $D := \delta_+^N$ and $U := \delta_4^N/(2d)$. Imposing our final condition
\begin{equation} \label{finc}
p^c > (\delta_+/\delta_4)^d
\end{equation}
on our parameters $\delta_1,\ldots,\delta_4,\delta_+$, and $c$, the conclusion is that $P$ lies in a torsion translate of a proper sub $t$-module $H$. Since in our case $\kappa = \delta$, part (ii) of Conjecture~\ref{capconj} and the general inequality $\kappa \leq \delta$ imply that the saturation and dynamic degrees of $H$ are still equal to $\kappa = \delta$. Continuing, $P$ is forced to be preperiodic.

It remains to note that the conditions~(\ref{deltprim}), (\ref{compr}) and (\ref{finc}) are compatible precisely when~(\ref{contpo}) holds. \proofend

\end{document}